\documentclass{article}

\usepackage{latexsym,upref,enumerate}
\usepackage[english]{babel}
\usepackage{amsthm}
\usepackage{amssymb}
\usepackage{enumitem} 
\usepackage{nicefrac}
\usepackage{mathtools}
\usepackage{cases}
\usepackage[final]{microtype} 
\usepackage{color}
\usepackage[pdftex,dvipsnames]{xcolor}  
\usepackage[colorinlistoftodos,prependcaption,textsize=tiny]{todonotes}
\usepackage{xargs}

\newtheorem{theorem}{Theorem}[section]
\newtheorem{proposition}{Proposition}[section]
\newtheorem{lemma}{Lemma}[section]

\usepackage{subcaption}
\usepackage{cite}
\usepackage{authblk}

\newcommand{\bb}{\begin{equation}}
\newcommand{\ee}{\end{equation}}
\newcommand{\ba}{\begin{array}}
\newcommand{\ea}{\end{array}}

\usepackage[all]{xy}

\numberwithin{equation}{section}

\usepackage{ulem}

\title{On the Cauchy Problem for the Dispersion Generalized Camassa-Holm Equation}

\author[1 ] {Nesibe Ayhan}
\author[2 ]{Nilay Duruk Mutluba\c{s}}
\affil[1]{Karl-Franzens-Universität Graz \\Institute for Mathematics and Scientific Computing\\ Graz, Austria\\ 
\texttt{nesibe.ayhan@uni-graz.at}}

\affil[2]{Sabanc{\i} University \\Faculty of Engineering and Natural Sciences \\Istanbul, Turkey\\
\texttt{nilay.duruk@sabanciuniv.edu}}

\begin{document}

\maketitle

\begin{center}

\textbf{Abstract}
\end{center}

In this paper, we establish local well-posedness of the Cauchy problem for a recently proposed dispersion generalized Camassa-Holm equation by using Kato's semigroup approach for quasi-linear evolution equations. We show that  for initial data in the Sobolev space $H^{s}(\mathbb{R})$ with $s>\frac{7}{2}+p$, the  Cauchy problem is locally well-posed, where $p$ is an even real number determined by the order of the positive differential operator $L$ corresponding to the dispersive effect added to the Camassa-Holm equation.


{\bf MSC classification 2010:} {Primary: 35G25, 35L30, 47D60; Secondary: 35B65, 47D03. }



\section{Introduction}
The nonlinear dispersive wave equation
\begin{equation}\label{eqn1}
    u_{t} - u_{xxt} + 3uu_{x} = 2u_{x}u_{xx} + uu_{xxx},
\end{equation}
was introduced by Camassa and Holm \cite{camassa} to model the unidirectional propagation of shallow water waves over a flat bottom. Here, \(u(x, t)\) represents the fluid velocity at time \(t\) and in the spatial direction \(x\). The equation is known as the Camassa–Holm (CH) equation, whose various generalizations have appeared in the literature in recent years. These generalizations are usually done qualitatively on the structure of the equation, without paying much attention on the physical meaning and derivation. 

The Korteweg deVries (KdV) and Benjamin-Bona-Mahony (BBM) equations are canonical wave equations which characterizes unidirectional dispersive wave propagation in continuum having different physical properties. Although the CH equation was proposed primarily to model propagation of waves of moderate amplitude in shallow water, it is expected to have the equation arising as a model equation in different continuum where dispersive and nonlinear affects are balanced properly. As a matter of fact there are studies in this direction (for instance, see \cite{dai} about the nonlinear waves propagating along a  circular cylindrical rod composed of a compressible elastic  material and \cite{chen} about the nonlinear waves propagating in  a pre-stressed, thin elastic plate composed of a compressible hyperelastic material). Since linear dispersive affect can occur in different forms for different continuum due to various reasons such as  nonlocal effects, inhomogeneity, anisotropy, multidimensionality, etc and this changes the inertia term $u_t-u_{xxt}$, having generalized form of inertia term, writing the corresponding equation and analyzing the new model are interesting research problems. Therefore, in this paper, we focus on the dispersive nature of CH equation and analyze how the behavior of the solutions change when the dispersive effect is generalized.

When the terms in the CH equation are examined in detail, the \(uu_{x}\) term denotes nonlinear steeping, the \(u_{xxt}\) term denotes the linear dispersion effect, and the \(2u_{x}u_{xx}+uu_{xxx}\) terms denote the nonlinear dispersion effect. When the momentum density \(m =(1-\partial_{x}^2)u\) is defined, the CH equation becomes
\begin{equation}\label{md-gen}
        m_{t}+m_{x}u+2mu_{x}=0.
\end{equation}

There are generalizations of the CH equation for different momentum density forms in the literature. Important examples of these can be given as follows:\\
\\
\((i)\)\ Hunter and Saxton in \cite{hunter1991dynamics} considered \(m=-\partial_{x}^{2}u\).\\
\\
\((ii)\)\ Holm et al. in \cite{holm2009singular} considered 
   $m=(1-\alpha^{2}\partial_{x}^{2})u$, where \(\alpha\) is a constant.\\
\\
\((iii)\)\ Khesin et al. in \cite{khesin2008generalized} introduced a \(\mu\)-version of Camassa-Holm equation as follows
\begin{center}
    \(m_{t}+2mu_{x}+m_{x}u=0,\ \ \ \ m=(\mu-\partial_{x}^{2})u\),
\end{center}

where $u(x,t)$ is a time-dependent function on the unit circle $\mathbb{S}=\mathbb{R}/\mathbb{Z}$ and \mbox{\(\mu(u)=\int_{\mathbb{S}}u dx\)} denotes its mean. This equation describes evolution of rotators in liquid crystals with external magnetic field and self-interaction. It is also studied in \cite{deng2020global}, \cite{fu2012blow}, \cite{wang2022blow}, \cite{Yamane2020-nk}. \\ For the periodic case, Wang \cite{wang2018well} considered $m=\mu(u)-u_{xx}+u_{xxxx}$. Moreover, Wang studied the modified \(\mu\)-version of Camassa-Holm equation in \cite{wang2016non} as follows
\begin{center}
    \(m_{t}+(2\mu(u)u-u_{x}^2)m)_{x}=0\).
\end{center}
\((iv)\)\ Wang considered \(m=\mu(u)+u_{xxxx}\) in \cite{wang2017-za}.\\
\\
\((v)\)\ Ding et al. considered \(m=u-u_{xx}+u_{xxxx}\) in \cite{ding2017lipschitz}.\\
\\
\((vi)\)\ There are also higher order forms of the CH equation, where \(m=(1-\partial_{x}^2)^{k}u\) for positive integer \(k\) \cite{constantin2003geodesic}, which describe exponential curves of the manifold of smooth orientation-preserving diffeomorphisms of the unit circle in the plane.  Studies for \(k=2\) can be found in \cite{mclachlan2009well}, \cite{mu2011well}, \cite{reyes2021pseudo}.\\
\\
\((vii)\)\ For \(r\geq 1\), Camassa-Holm system with two components, where \(m=(1-\partial_{x}^2)^{r}u\) is studied in \cite{chen2017well}.\\

The common feature of the examples in the literature is that they put \mbox{different} effects on linear and non-linear dispersive terms to observe the results. In this paper, our main aim is to study
\begin{equation}\label{md-gen}
    m_{t}+bm_{x}u+amu_{x}=0,\quad a,b >0,
\end{equation}
with the following form of momentum density:
\begin{equation}\label{new-md}
       m=(1-L\partial_{x}^2)u.
\end{equation}

\noindent
Here, $L$ is a general differential operator in spatial variable $x$ whose order is an even real number \(p\). With this momentum density, the dispersive effect in (\ref{eqn1}) is generalized in (\ref{md-gen}).
Note that  choosing \(L\) as the identity operator with \(a=2\) and \(b=1\), (\ref{md-gen}) corresponds to the CH equation given by \((\ref{eqn1})\). Furthermore, choosing $L$ as $\alpha^2$ multiple of identity operator corresponds to the example \((ii)\), as $Id-\partial_x^2$ corresponds to the example \((v)\), as $2-\partial_x^2$ corresponds to the example \((vi)\).\\

The CH equation with the generalized dispersion feature defined in (\ref{new-md}) has been proposed for the first time in the literature. In the equation (\ref{md-gen}),  both the linear and the non-linear dispersion effects have changed. In terms of qualitative analysis, this makes a big difference since $L$ operator is given in closed form and applies also on nonlinear terms in the equation. 

This paper presents the dispersion generalized Camassa-Holm equation given by (\ref{md-gen})-(\ref{new-md}) and proves local well-posedness of the solutions for the corresponding Cauchy problem. It can be considered as a non-local and nonlinear dispersive partial differential equation and a mathematical generalization of the classical Camassa-Holm equation rather than a physical generalization. We prove that the Cauchy problem is locally well-posed on the real line for the initial data in Sobolev space $H^{s}(\mathbb{R})=H^s$, \(s> \frac{7}{2}+p\). The even real \mbox{number} \(p\) is the order of the general differential operator \(L\) which appears in closed form. 

Our proof is based on Kato's semigroup approach for quasilinear equations. For this reason, in Section \ref{lwp} we present a short review of the theorem we rely on and establish local well-posedness for the dispersion generalized Camassa-Holm equation. We make use of commutator operators to obtain a \mbox{suitable} form of the equation to use Kato's semigroup approach. Main reason for rewriting is that \(L\) is in closed form among the nonlinear terms in the equation as well and usual \mbox{differentiation} rules do not apply. 


In Section  \ref{comp}, we compare the results for Camassa-Holm type equations with those of the dispersion generalized Camassa-Holm equation. The changes in the dispersive effect, the differences in their non-local forms, the initial data classes chosen for the corresponding Cauchy problems are discussed.

We end the paper with Section \ref{open} in which we provide open problems to be \mbox{discussed}. According to mathematical analysis questions for the Camassa-Holm equation \mbox{appearing} in the literature, we can say that further qualitative analysis is also possible for the dispersion generalized Camassa-Holm equation, such as  global well-posedness and finite time blow-up.

\section{Local Well-posedness} \label{lwp}
\subsection{Semigroup Approach}\label{Katotheory}

\noindent
Consider the abstract quasi-linear evolution equation in the Hilbert space $X$:
  \begin{equation}
  \label{qlee}
     u_t +A(u)u = f(u), ~~~~t\geq 0,~~~~~u(0)=u_0.
  \end{equation}
Let $Y$ be a second Hilbert space such that $Y$ is continuously and densely injected into $X$ and let $S:Y\rightarrow X$ be a topological isomorphism. Assume that

\begin{itemize}
\item[(A1)]  For any given $r>0$ it holds that for all $u \in  \mathrm B_r(0) \subseteq Y$ (the ball around the origin in $Y$ with radius $r$), the linear operator $A(u)\colon X \to X$ generates a strongly continuous semigroup $T_u(t)$ in $X$ which satisfies
\begin{equation*}
\| T_u(t) \|_{\mathcal L(X)} \leq \mathrm e^{\omega_r t} \quad \text{for all} \quad t\in [0,\infty)
\end{equation*}
for a uniform constant $\omega_r > 0$.
\item[(A2)] $A$ maps $Y$ into $\mathcal L(Y,X)$, more precisely the domain $D(A(u))$ contains $Y$ and the restriction $A(u)|_Y$ belongs to $\mathcal L(Y,X)$ for any $u\in Y$. Furthermore $A$ is Lipschitz continuous in the sense that for all $r>0$ there exists a constant $C_1$ which only depends on $r$ such that
\begin{equation*}
\| A(u) - A(v) \|_{\mathcal L(Y,X)} \leq C_1 \, \|u-v\|_X 
\end{equation*}
for all $u,~v$ inside $\mathrm B_r(0) \subseteq Y$.
\item[(A3)] For any $u\in Y$ there exists a bounded linear operator $B(u) \in \mathcal L(X)$ satisfying $B(u) = S A(u) S^{-1} - A(u)$ and $B \colon Y \to \mathcal L(X)$ is uniformly bounded on bounded sets in $Y$. Furthermore for all $r>0$ there exists a constant $C_2$ which depends only on $r$ such that  
\begin{equation*}
\| B(u) - B(v)\|_{\mathcal L(X)} \leq C_2 \, \|u-v\|_Y
\end{equation*}
for all $u,~v \in \mathrm B_r(0)\subseteq Y$.
\item[(A4)] 
For all $t\in[0,\infty)$, $f$ is uniformly bounded on bounded sets in $Y$. Moreover, the map $f\colon Y \to Y$ is locally $X$-Lipschitz continuous in the sense that for every $r>0$ there exists a constant $C_3>0$, depending only on $r$, such that
\begin{equation*}
\| f(u) - f(v)\|_{X} \leq C_3 \, \|u-v\|_X \quad \text{for all} \; u,~v \in \mathrm B_r(0) \subseteq Y
\end{equation*}
and locally $Y$-Lipschitz continuous in the sense that for every $r>0$ there exists a constant $C_4>0$, depending only on $r$, such that
\begin{equation*}
\| f(u) - f(v)\|_{Y} \leq C_4 \, \|u-v\|_Y \quad \text{for all} \; u,~v \in \mathrm B_r(0) \subseteq Y.
\end{equation*}
\end{itemize}

\begin{theorem} \cite{KatoI}
\label{kato}
Assume that (A1)-(A4) hold. Then for given $u_0\in Y$, there is a maximal time of existence $T>0$, depending on $u_0$, and a unique solution $u$ to (\ref{qlee}) in $X$ such that
\begin{displaymath}
u=u(u_0,.)\in C([0,T),Y)\cap C^1([0,T),X).
\end{displaymath}
Moreover, the solution depends continuously on the initial data,\\
i.e. the map $u_0\rightarrow u(u_0,.) $ is continuous from $Y$ to \mbox{$C([0,T),Y)\cap C^1([0,T),X)$}.
\end{theorem}

\subsection{Local Well-posedness Theorem}

\noindent
In this subsection, we apply Kato's semigroup approach to establish local \mbox{well-posedness} for the Cauchy problem associated to the generalized Camassa-Holm equation:
\begin{equation}
\begin{cases}
        u_{t}-L\partial^2_{x}u_{t}+(a+b)uu_{x}-au_{x}L\partial^2_{x}u-buL\partial^2_{x}u_{x}=0, & t>0, \ x\in\mathbb{R},\label{2.2}\\
u(x,0)=u_{0}(x), & x\in \mathbb{R},\\
\end{cases}
\end{equation}
where \(a\) and \(b\) are positive constants, \(L\) is a positive operator with even order \(p\). Note that we rewrite the equation (\ref{md-gen}) for which (\ref{new-md}) is valid. To construct a non-local form of this equation, we use the usual commutator of two operators $[.,.]$:
\begin{center}
    \([L\partial^2_{x},u]u_{x}=L\partial^2_{x}(uu_{x})-uL\partial^2_{x}u_{x}\).
\end{center}
Then,
\begin{align*}
    (a+b)[L\partial^2_{x}, u]u_{x} &=\ (a+b)L\partial^2_{x}(uu_{x})-(a+b)uL\partial^2_{x}u_{x}&\\
    \\
    &=\ (a+b)L\partial^2_{x}(uu_{x})-auL\partial^2_{x}u_{x}-buL\partial^2_{x}u_{x}&.\\
\end{align*}
So, we can write
\begin{center}
    \((a+b)uu_{x}-buL\partial^2_{x}u_{x}=(a+b)(1-L\partial^2_{x})(uu_{x})+(a+b)[L\partial^2_{x}, u]u_{x}+auL\partial^2_{x}u_{x}\).
\end{center}
Then, we have
\begin{center}
    \((1-L\partial^2_{x})u_{t}+(a+b)(1-L\partial^2_{x})(uu_{x})+(a+b)[L\partial^2_{x}, u]u_{x}+auL\partial^2_{x}u_{x}-au_{x}L\partial^2_{x}u=0\).
\end{center}
Also, it can be seen that
\begin{align*}
    auL\partial^2_{x}u_{x}-au_{x}L\partial^2_{x}u &= a(uL\partial^2_{x}u_{x}-u_{x}L\partial^2_{x}u)\\
    \\
    &=a(uL\partial^2_{x}u_{x}-(L\partial^2_{x}u)u_{x})\\
    \\
    &=a(uL\partial^2_{x}-(L\partial^2_{x}u))u_{x}\\
    \\
    &=-a[L\partial^2_{x},u]u_{x}.
\end{align*}
Then, the equation (\ref{2.2}) becomes
\begin{center}
    \((1-L\partial^2_{x})u_{t}+(a+b)(1-L\partial^2_{x})(uu_{x})+b[L\partial^2_{x}, u]u_{x}=0\).
\end{center}
With \( \Gamma^{s} = (1-L\partial^2_{x})^{s/p+2}\), this equation takes the following quasi-linear form:
\begin{center}
    \(u_{t}+A(u)u=f(u)\),
\end{center}
where
\begin{center}
    \begin{equation}
        A(u)=(a+b)u\partial_{x}+b\Gamma^{-(p+2)}[L\partial^2_{x}, u]\partial_{x}=r(u)\partial_{x}\label{A},
    \end{equation}
\end{center}
and
\begin{center}
\begin{equation}
    f(u)=0\label{f}.
\end{equation}
\end{center}
Here, notice that the operator \(L\) is in a closed form, and even though there are more than one possible way of writing non-local form of an equation, the form where we collect nonlinear effects in the operator \(A(u)\), as above, holds and serves for our purpose. Furthermore, recalling the approach, we choose the spaces $X:=(H^{s-1},||.||_{s-1})$, $Y:=(H^{s},||.||_{s})$ and consider the topological isomorphism $S\coloneqq\Gamma=(1-L\partial^2_{x})^{1/p+2}\colon Y\to X$ between these spaces, which defines an isometry, 
i.e.~$\| \Gamma u \|_{s-1} = \|u\|_s$ for all $u\in H^{s}$.

\noindent
Hence, the main result of this paper is the following:

\begin{theorem}\label{main}
Assume that assumptions \((A1)-(A4)\) hold for (\ref{2.2})-(\ref{f}). Let $u_{0}\in H^s$, $s>\frac{7}{2}+p$ be given. Then, there exists a maximal time of existence \(T>0\), depending on \(u_{0}\), such that there is a unique solution \(u\) satisfying
\begin{center}
    \(u \in C([0, T), H^s) \cap C^{1}([0, T), H^{s-1})\).
\end{center}
Moreover, the solution depends continuously on the initial data, i.e, the map \(u_{0} \rightarrow u(u_{0}, .)\) is continuous from \(H^{s}\) to \(C([0, T), H^s) \cap C^{1}([0,T), H^{s-1})\).\\
\end{theorem}

In order to prove this result, we will apply Kato's semigroup approach. Since \(f(u)=0\) in our Cauchy problem, we only need to verify assumptions (A1)-(A3).\\
\\
Note that if \(L\) is the identity operator, \(a=2\) and \(b=1\), we get the Camassa-Holm equation. Considering this, the steps in the proof can be followed clearly.
 
\subsection*{Proof of Assumption (A1):}
Below, you will find the lemmas to be used in the proof of assumption \((A1)\).\\

\begin{lemma}\label{L2}
The operator \(A(u)=r(u)\partial_{x}\) in \(L^2\), with \(u \in H^{s}, s>\frac{7}{2}+p\) is quasi-m-accretive.
\end{lemma}
\begin{proof}
Since \(L^2\) is a Hilbert space with standard inner product $(.,.)_0$, \(A(u)\) is quasi-m-accretive if and only if there is a real number \(\beta\) such that
\begin{itemize}
    \item[(a)] \((A(u)w, w)_{0} \geq -\beta||w||_{0}^2\),
    \item[(b)]  The range of \(A(u)+\lambda\) is all of \(X\) for some (or all) \(\lambda>\beta\).
\end{itemize}

\noindent
First, we will prove part \((a)\). By using integration by parts, the left-hand side of the equality can be written as follows:
\begin{center}
    \((A(u)w, w)_{0} = (r(u)\partial_{x}w, w)_{0} = \frac{-1}{2}((r(u))_{x}w, w)_{0} \)
\end{center}
since if we let
\begin{align*}
K=(r(u)w_{x},\ w)_{0}\ &= -(r(u)w_{x},\ w)_{0}\ -\ ((r(u))_{x}w,\ w)_{0},&
\end{align*}
where \(-(r(u)w_{x},\ w)_{0} = -K\). Then,
\begin{align*}
2K = - ((r(u))_{x}w,\ w)_{0} \\
\\
K = -\frac{1}{2} ((r(u))_{x}w,\ w)_{0}.
\end{align*}
%
%
Then, it follows that
\begin{align*}
|(r(u)\partial_{x}w, w)_{0}| &= |\frac{-1}{2}((r(u))_{x}w, w)_{0}|&\\
\\
&\leq c||(r(u))_{x}w||_{0}||w||_{0}&\\
\\
&\leq c||(r(u))_{x}||_{L^{\infty}}||w||_{0}^{2}&\\
\\
&\leq c||(r(u)||_{s}||w||_{0}^{2}&\\
\\
&\leq \beta||w||_{0}^{2},
\end{align*}
where $\beta$ is a constant depending on $||u||_s$. Since \(u \in H^{s}\) with \(s>\frac{7}{2}+p\), it follows that \(||u_{x}||_{L^{\infty}} \leq ||u||_{s}\). We show that the operator \(A\) is dissipative for all \(\lambda>\beta\). That means the operator \(-A\) is accretive, where we have \((A(u)w, w)_{0} \geq -\beta||w||_{0}^2\) as required.\\
\\
Now, we will prove part \((b)\). Note that if \(A\) is a closed operator, then \(A(u)+\lambda\) has closed range in \(X\) for all \(\lambda>\beta\). So, it is enough to show that \(A(u)+\lambda\) has dense range in \(x\) for all  \(\lambda>\beta\).\\
\\First, we will show that \(A\) is a closed operator in \(L^2\). Let \((v_{n})\) be a \mbox{sequence} in \(\mathcal{D}(A)\) which converges to \(v\in L^2\) and \(Av_{n}\) converges to \(w\in L^2\). Then, since \(v_{n} \in \mathcal{D}(A)\) and \(\mathcal{D}(A)=\{w\in L^2 \ | \ r(u)w \in H^1\}\subset L^2\), we can conclude that \(rv_{n} \in H^1\). Also, by the continuity of the multiplication \(H^r\times L^2 \rightarrow L^2 \) for \(r > \frac{1}{2}\), both \(rv_{n} \rightarrow rv\) and \(r_{x}v_{n} \rightarrow r_{x}v\) in \(L^2\), which implies that \((rv_{n})_{x} \rightarrow w + r_{x}v \) in \(L^2\). Then, we have the sequences \((rv_{n})\) and \((rv_{n})_{x}\) converges in \(L^2\). Then we can conclude that \((rv_{n})\) converges to \(rv\) in \(H^1\), which implies that \(v \in \mathcal{D}(A) \). Moreover, by continuity of \(\partial_{x}: H^1 \rightarrow L^2\) implies that \(\lim_{n \to \infty}(rv_{n})_{x} = (rv)_{x}\), thus we get \(w = (rv)_{x} - r_{x}v = Av \). Hence, by definition, we showed that A is a closed operator.\\
\\
Now, we need to show that \((A(u)+\lambda)\) has dense range in \(L^2\) for all \(\lambda>\beta\). Note that the adjoint operator of the \(A(u)=r(u)\partial_{x}\) can be written as
\begin{center}
    \(A^{*}(u) = -r_{x}(u)-r(u)\partial_{x}\).
\end{center}
Then,
\begin{center}
    \(A^{*}(u)w = -r_{x}(u)w-r(u)w_{x} = -(r(u)w)_{x}\),
\end{center}
where \(r_{x}(u)w \in L^2 \) since \(u_{x} \in L^{\infty}\) and \(w \in L^2\), and \(r(u)w_{x}=A(u) \in L^2\) for \mbox{\(w \in \mathcal{D}(A)\)}. Hence, we can obtain that
\begin{center}
    \(\mathcal{D}(A^{*})=\{w\in L^2 \ | \ A^{*}(u)w \in L^2\}\).
\end{center}
On the contrary, assume that the range of \((A(u)+\lambda)\) is not all of \(L^{2}\). Then, there exists \(z\neq 0 \in L^{2}\) such that
\begin{center}
    \((A(u)w, z)_{0}=0\), \ \(\forall w \in \mathcal{D}(A(u))\).
\end{center}
Since \(H^1\subset \mathcal{D}(A)\), \(\mathcal{D}(A)\) is dense in \(L^2\). Then, due to \(\mathcal{D}(A^{*})\) is closed, \(z\in \mathcal{D}(A^{*})\). Then, by using the fact that \(\mathcal{D}(A)=\mathcal{D}(A^{*})\), we can write
\begin{center}
    \(((A(u)+\lambda)w, z)_{0}=(w, (A(u)+\lambda)^*z)_{0}=0\),
\end{center}
which implies that \((A(u)^{*}+ \lambda )z = 0\) in \(L^{2}\). After multiplying this equality by \(z\), we can rewrite it as
\begin{align*}
    0=((A^{*}(u)+\lambda)z, z)_{0} &= (A^{*}(u)z, z)_{0}+(\lambda z, z)_{0} &\\
    \\
    &=(z, A(u)z)_{0} + (\lambda z, z)_{0}&\\
    \\
    &\geq(\lambda-\beta)||z||_{0}^2, \ \ \ \forall \lambda > \beta.&
\end{align*}
Since for all \(\lambda > \beta\), the term \((\lambda-\beta)>0\). Therefore, \(z=0\). However, it contradicts with the assumption \(z\neq0\), which completes the proof of Lemma \ref{L2}.\\
\end{proof}

\noindent
Now, we give the commutator estimate needed for the upcoming lemma:
\begin{proposition} [\cite{lannes})]\label{ctt}
    Let \(n>0\), \(s \geq 0 \), and \(3/2 < s+n \leq \sigma\). Then, for all \(f \in H^{\sigma}\) and \(g \in H^{s+n-1}\), one has
\begin{center}
    \( ||[\Lambda^{n},\ f]g||_{s} \leq c||f||_{\sigma}||g||_{s+n-1}\),
\end{center}
where \(c\) is a constant which is independent of \(f\) and \(g\).
\end{proposition}

\begin{lemma}\label{s-1}

The operator \(A(u)=r(u)\partial_{x}\) in \(H^{s-1}\), with \(u \in H^{s}, s>\frac{7}{2}+p\) is quasi-m-accretive.

\end{lemma}

\begin{proof}

Since $H^{s-1}$ is a Hilbert space, $A(u)=r(u)\partial_{x}$ is quasi-m-accretive if and only if there is a real number $\beta$ such that
\begin{itemize}
    \item [$(a)$] $(A(u)w, w)_{s-1} \geq -\beta||w||_{s-1}^2,$
    \item [$(b)$] -$A(u)$ is the infinitesimal generator of a $C_{0}$-semigroup on $H^{s-1}$, for some (or all) $\lambda>\beta$.
\end{itemize}

First, we will prove part $(a)$. Since \(u \in H^{s}\), with \(s>\frac{7}{2}\), we can say that \(u\) and \(u_{x}\) belong to $L^{\infty}$. Then, it follows that \(||u_{x}||_{L^{\infty}} \leq ||u||_{s}\). Note that
\begin{align*}
\Gamma^{s-1}(r(u)\partial_{x}w) &= [\Gamma^{s-1}, r(u)]\partial_{x}w+r(u)\Gamma^{s-1}(\partial_{x}w)&\\
\\
&=\ [\Gamma^{s-1}, r(u)]\partial_{x}w+r(u)\partial_{x}\Gamma^{s-1}w.&\\
\end{align*}

Then, we have
\begin{align*}
(A(u)w, w)_{s-1} &=\ (r(u)\partial_{x}w,\ w)_{s-1} &\\
\\
&=\ (\Gamma^{s-1}r(u)\partial_{x}w,\ \Gamma^{s-1}w)_{0}&\\
\\
&=\ ([\Gamma^{s-1}, r(u)]\partial_{x}w,\ \Gamma^{s-1}w)_{0} + (r(u)\partial_{x}\Gamma^{s-1}w,\ \Lambda^{s-1}w)_{0}.&\\
\end{align*}

For the first term \(([\Gamma^{s-1}, r(u)]\partial_{x}w,\ \Gamma^{s-1}w)_{0}\), use the commutator estimate \mbox{(Proposition \ref{ctt})} with \(n=s-1\), and \(\sigma=s\). Then, we get
\begin{align*}
|([\Gamma^{s-1}, r(u)]\partial_{x}w,\ \Gamma^{s-1}w)_{0}| &\leq\ c||(r(u)||_{s}||\partial_{x}w||_{s-2}||w||_{s-1}&\\
\\
&\leq\ \Tilde{c}||w||^{2}_{s-1},
\end{align*}
where \(\Tilde{c}\) is a constant depending on \(||u||_{s}\).\\
\\
For the second term \((r(u)\partial_{x}\Gamma^{s-1}w,\ \Gamma^{s-1}w)_{0}\), use the integration by parts to get
\begin{align*}
|(r(u)\partial_{x}\Gamma^{s-1}w,\ \Gamma^{s-1}w)_{0}| &=\ \big|-\frac{1}{2}(r(u)_{x},\ \Gamma^{s-1}w)^2)_{0}\big|&\\
\\
&\leq\ c||r(u)_{x}||_{L^{\infty}}||\Gamma^{s-1}w||^{2}_{0}&\\
\\
&\leq\ c||r(u)_{x}||_{L^{\infty}}||w||^{2}_{s-1}&\\
\\
&\leq\ \Tilde{c}||w||^{2}_{s-1},
\end{align*}
where \(\Tilde{c}\) is a constant depending on \(||u||_{s}\).\\
\\
Set \(\beta=\Tilde{c}||u||_{s}\). Then, we get \((A(u)w, w)_{s-1} \geq -\beta||w||_{s-1}^2\), as required.\\
\\
Now, we will prove part \((b)\). Let \(Q=\Gamma^{s-1}\), note that \(Q\) is an isomorphism of \(H^{s-1}\) to \(L^{2}\), and \(H^{s-1}\) is continuously and densely embedded into \(L^2\) as \(s>\frac{3}{2}\). Define
\begin{align*}
    A_{1}(u):= QA(u)Q^{-1} &=\ \Gamma^{s-1} A(u) \Gamma^{1-s}&\\
    \\
    &=\ \Gamma^{s-1} r(u)\partial_{x} \Gamma^{1-s}&\\
    \\
    &=\ \Gamma^{s-1} r(u) \Gamma^{1-s}\partial_{x}&\\
\end{align*}
and \(B_{1}=\ A_{1}(u) +\ A(u)\).\\
Let \(w \in L^{2}\) and \(u \in H^{s}\), where \(s>\frac{7}{2}+p\). Then, we have
\begin{align*}
    ||B_{1}(u)||_{0} &=\ ||[\Gamma^{s-1},\ A(u)]\Gamma^{1-s}w||_{0}&\\
    \\
    &=\ ||[\Gamma^{s-1},\ r(u)]\Gamma^{1-s}\partial_{x}w||_{0}&\\
    \\
    &\leq\ c||r(u)||_{s}||\Gamma^{1-s}\partial_{x}w||_{s-2}&\\
    \\
    &\leq\ c||r(u)||_{s}||w||_{0}&,
\end{align*}
where we again use the commutator estimate (Proposition \ref{ctt}) with \(n=s-1\), and \(\sigma=s\). Therefore, we obtain \(B_{1}(u) \in \mathcal{L}(L^2)\).
\end{proof}

\begin{lemma} \cite{pazy2012semigroups}\label{semi}
Let \(X\) and \(Y\) be two Banach spaces such that \(Y\) is continuously and densely embedded in \(X\). Let -\(A\) be the infinitesimal generator of the \(C_{0}\)-semigroup \(T(t)\) on \(X\) and let \(Q\) be an isomorphism from \(Y\) onto \(X\). Then \(Y\) is -\(A\)-admissible (i.e. \(T (t)Y \subset Y\) for all \(t \geq 0\), and the restriction of \(T(t)\) to \(Y\) is a \(C_{0}\)-semigroup on \(Y\) ) if and only if -\(A_{1} =\) - \(QAQ^{-1}\) is the infinitesimal generator of the \(C_{0}\)-semigroup \(T_{1}(t) = QT(t)Q^{-1}\) on \(X\). Moreover, if \(Y\) is -\(A\)-admissible, then the part of -\(A\) in \(Y\) is the infinitesimal generator of the restriction \(T(t)\) to \(Y\).
\end{lemma}

We show that \(A(u)\) is quasi-m-accretive in \(L^{2}\), i.e, -\(A(u)\) is the infinitesimal generator of \(C_{0}\)-semigroup on \(H^{s-1}\) by Lemma \ref{s-1}. Thus, by using the perturbation theorem for semigroup \cite{pazy2012semigroups}, we can say that \(A_{1}(u) = A(u) + B_{1}(u)\) is also the infinitesimal generator of \(C_{0}\)-semigroup on \(L^{2}\). Then, we can conclude that \(H^{s-1}\) is -\(A\)-admissible. Hence, -\(A(u)\) is the infinitesimal generator of \(C_{0}\)-semigroup on \(H^{s-1}\) by Lemma \ref{semi} with \(X=L^2\), \(Y=H^{s-1}\), and \(Q=\Gamma^{s-1}\). This completes the proof of Lemma \ref{s-1}, and thus assumption \((A1)\).

\subsection*{Proof of Assumption (A2):}
Below, you will find the needed lemmas to be used in the proof of assumption \((A2)\).

\begin{lemma}\label{A2}
Let the operator \(A(u)=r(u)\partial_{x}\), with \(u \in H^{s}\), \(s > \frac{7}{2}+p\). Then, \(A(u) \in \mathcal{L}(H^{s}, H^{s-1})\), for any \(u \in H^{s}\). Moreover,
\begin{center}
    \(||(A(u)-A(v))w||_{s-1} \leq \lambda_{1} ||u-v||_{s-1}||w||_{s}\), \quad \(u,v,w \in H^{s}\).
\end{center}
\end{lemma}

\begin{proof}
Let \(u, v, w \in H^{s}\) with \(s>\frac{7}{2}+p\), and note that \(H^{s-1}\) is a Banach algebra.\\
Then, we have
\begin{align*}
||(A(u)-A(v))w||_{s-1} &= ||((a+b)(u-v)\partial_{x}+b\Gamma^{-(p+2)}[L\partial^2_{x}, (u-v)]\partial_{x})w||_{s-1}&\\
\\
&\leq ||(a+b)(u-v)\partial_{x}w||_{s-1}+||b\Gamma^{-(p+2)}[L\partial^2_{x}, (u-v)]\partial_{x}w||_{s-1}&\\
\\
&\leq ||(u-v)||_{s-1}||\partial_{x}w||_{s-1}+||[L\partial^2_{x}, (u-v)]\partial_{x}w||_{s-p-3}&.
\end{align*}

We will use the commutator estimate (Proposition \ref{ctt}) with \(n=p+2\), \(s=s-p-3\), and \(\sigma=s-1\), which implies \(s+n-1=s-2\). Then, for \(f=u-v\) and \(g=\partial_{x}w\), we get
\begin{align*}
||[L\partial^2_{x}, (u-v)]\partial_{x}w||_{s-p-3} &\leq c||(u-v)||_{s-1}||\partial_{x}w||_{s-2}&\\
\\
&\leq c||u-v||_{s-1}||w||_{s-1}&\\
\\
&\leq \lambda_{1}||u-v||_{s-1}||w||_{s},&
\end{align*}
where \(\lambda_{1}\) is a constant. Then, we get
\begin{center}
    \(||(A(u)-A(v))w||_{s-1} \leq \lambda_{1}||u-v||_{s-1}||w||_{s} \).
\end{center}
Take \(v=0\) in the above inequality to obtain \(A(u) \in \mathcal{L}(H^{s},\ H^{s-1})\). This completes the proof of Lemma \ref{A2}, and thus assumption \((A2)\).

\end{proof}

\subsection*{Proof of Assumption (A3):}
Below, you will find the needed lemmas to be used in the proof of assumption \((A3)\).

\begin{lemma}\label{A3}
For any \(u\in H^{s}\), there exists a bounded linear operator \(B(u) \in \mathcal{L}(H^{s-1})\) satisfying \(B(u) = \Gamma A(u)\Gamma^{-1}-A(u)\), where \(B:H^{s} \rightarrow \mathcal{L}(H^{s-1})\) is uniformly bounded sets in \(H^{s-1}\). Moreover,
\begin{center}
    \( ||(B(u)-B(v))w||_{s-1} \leq \lambda_{2}||u-v||_{s}||w||_{s-1} \), \quad \(u,v \in H^{s}\), \(w\in H^{s-1}\).
\end{center}
\end{lemma}
\begin{proof}
Note that since \(\partial_{x}\) and \(\Gamma\) commute, we can rewrite \(B(u)\) as
\begin{center}
\(B(u)=\Gamma A(u)\Gamma^{-1}-A(u)=\Gamma r(u)\partial_{x}\Gamma^{-1}-r(u)\partial_{x}=[\Gamma, r(u)]\Gamma^{-1}\partial_{x}\).
\end{center}
First, we will show that \(B(u)\) is bounded. To do that again we will use the \mbox{commutator} estimate (Proposition \ref{ctt}) with \(n=1\), \(s=s-1\), and \(\sigma=s-1\), which implies \(s+n-1=s-1\). Then, for \(f=r(u)\) and \(g=\Gamma^{-1}\partial_{x}w\), where \(w\in H^{s-1}\), we can write
\begin{align*}
||B(u)w||_{s-1} &= ||[\Gamma, r(u)]\Gamma^{-1}\partial_{x}w||_{s-1}&\\
\\
&\leq ||r(u)||_{s}||\Gamma^{-1}\partial_{x}w||_{s-1}&\\
\\
&\leq ||r(u)||_{s}||w||_{s-1}&\\
\\
 &\leq c||w||_{s-1},&
\end{align*}
where \(c\) depends on \(||u||_{s}\).\\
\\
Moreover,
\begin{align*}
||(B(u)-B(v))w||_{s-1} &= ||[\Gamma, r(u)-r(v)]\Gamma^{-1}\partial_{x}w||_{s-1}&\\
\\
&\leq ||r(u)-r(v)||_{s}||\Gamma^{-1}\partial_{x}w||_{s-1}&\\
\\
&\leq ||r(u)-r(v)||_{s}||w||_{s-1}&\\
\\
&\leq||(a+b)(u-v)+a\Gamma^{-(p+2)}[L\partial_{x}^{2},  u-v]||_{s}||w||_{s-1}&\\
\\
&\leq \big(||u-v||_{s} + || [L\partial_{x}^{2}, u-v]_{s-p-2}\big) ||w||_{s-1}&\\
\\
&\leq \big(||u-v||_{s} + ||u-v||_{s}\big) ||w||_{s-1}&\\
\\
&\leq \lambda_{2}||w||_{s-1},&
\end{align*}
where \(\lambda_{2}\) is a constant depending on \(||u||_{s}\) and \(||v||_{s}\). Take \(v=0\) in the above inequality to obtain \(B(u) \in \mathcal{L}(H^{s-1})\). This completes the proof of Lemma \ref{A3}, and thus assumption \((A3)\).
\end{proof}

\subsection*{Proof of Assumption (A4):}
Since \(f(u)=0\), it is trivial.\\
\\
As we verify all the assumptions \((A1)-(A4)\) in Theorem \ref{main}, local well-posedness for the dispersion generalized Camassa-Holm equation is established.

\section{Comparison of Regularity}\label{comp}

As we mentioned before, there are numerous studies in which different forms of momentum density appear. Starting from the Camassa-Holm equation, we will refer to some of the regularity results for the solutions of Cauchy problems associated with these different types of evolution equations. We will verify that the result we provide in this paper is consistent with the ones appeared in the literature. Generalizing the dispersive effect also generalizes the degree of the regularity obtained for the solutions. 

Below, we give sample cases:
\begin{itemize}
    \item \textbf{Case 1:} \cite{rodriguez2001cauchy} $m=1-\partial_x^2$ which leads to the classical Camassa-Holm equation for which $L$ is the identity operator, local well-posedness is proved for $s>3/2.$
    \item \textbf{Case 2:} \cite{mu2011well} $m=(1-\partial_x^2)^2 u$, which is a special case of L with $p=2$, local well-posedness is proved for $s> 7/2.$
    \item \textbf{Case 3:} \cite{wang2017-za} $m=(\mu-\partial_x^2+\partial_x^4)u$, which is a special case of L with $p=2$, local well-posedness is proved for $s>7/2.$
\end{itemize}
Checking these results, we can confirm that the result we found in this paper provide the sufficient regularity of the solution corresponding to the Cauchy problem with initial data $u_0\in H^s$ and it generalizes the results previously reported in the literature.

\section{Discussion}\label{open}
Working on one of the main qualitative analysis questions related to the dispersive equation proposed in the present work, we can declare follow up questions as future work:
\begin{itemize}
    \item [1.] Can we find globally well-posed solution?
    \item [2.] Are there any conserved quantities? Is the energy conserved as it is for the Camassa-Holm equation?
    \item [3.] Is there a finite time at which blow-up occurs? If there is finite-time blow up, which norm of $u$ becomes unbounded? Is it in the form of wave breaking as it is for the Camassa-Holm equation? How does the change in the dispersive effect change the blow-up time?
\end{itemize}
\section{Conclusion}\label{conc}

In this study, we qualitatively generalize the dispersive effect through the generalized Camassa-Holm equation and establish the local well-posedness of the corresponding Cauchy problem by using Kato's semigroup approach. There are various studies in the literature for Camassa-Holm equation since it is a dispersive equation which can model wave breaking in shallow water wave theory. The 2:1 ratio of the coefficients corresponding to nonlinear terms enables to write non-local form of CH equation in a simple manner. In our case, having the operator $L$ in a closed form and being applied on nonlinear terms as well makes things difficult. 
 The operator \(L\) with even order \(p\) represents the generalization of the dispersive effect. Choosing $L$ as the identity operator and constants $a=2$, $b=1$, the equation reduces to Camassa-Holm equation. We obtain the results by making assumptions only on the order of \(L\). So, getting those results without writing the operator \(L\) explicitly will enable to evaluate the results of the equation in a very wide class. After a large number of trials, we observe that it is needed to write the quasi-linear form of the equation by collecting all nonlinear terms in the operator \(A(u)\). Therefore, in this case, \(f\) becomes zero. Thus, it is enough to verify the assumptions \((A1)-(A3)\) of the theorem. At the end, we see that choosing the initial data \(u_{0}\) from  $H^{s}$,\ where \(s>\frac{7}{2}+p\), the local well-posedness is established. One can observe that the initial data class needs to be more regular compared to Camassa-Holm equation. 
 
\section{Acknowledgement}
N. D. Mutluba\c{s} is supported by the Turkish Academy of Sciences within the framework of the Outstanding Young Scientists Awards Program (T\"{U}BA-GEBIP-2022).

\end{document}